\documentclass[11pt]{amsart}

\usepackage{amsmath}
\usepackage{amssymb}
\usepackage{colonequals}
\usepackage{graphicx}
\usepackage{xcolor}
\usepackage{hyperref}
\hypersetup{colorlinks=true, linkcolor=blue, citecolor=blue, urlcolor=magenta}

\def\F{\mathbb{F}}
\def\Z{\mathbb{Z}}
\def\R{\mathbb{R}}
\def\C{\mathbb{C}}
\def\Q{\mathbb{Q}}
\def\O{\mathcal{O}}
\def\eps{\varepsilon}

\DeclareMathOperator{\Ht}{Ht}

\newtheorem{theorem}{Theorem}[section]
\newtheorem{lemma}[theorem]{Lemma}
\newtheorem{corollary}[theorem]{Corollary}
\newtheorem{remark}[theorem]{Remark}
\numberwithin{equation}{section}
\newtheorem{prop}[theorem]{Proposition}

\title{Polynomial densities and Heilbronn's criterion}
\date{September 2025}

\author{Alexis Hibbler}
\address{Acera School\\Winchester, MA 01890, USA}
\email{alexis@aceraschool.org}

\author{Kevin J. McGown}
\address{Department of Mathematics and Statistics\\California State University, Chico\\Chico, CA 95929, USA}
\email{kmcgown@csuchico.edu}

\author{Enrique Trevi\~no}
\address{Department of Mathematics and Computer Science\\Lake Forest College\\Lake Forest, IL 60045, USA}
\email{trevino@lakeforest.edu}


\begin{document}

\begin{abstract}
Heilbronn gave a sufficient condition for a number
field with a totally ramified prime to fail to be norm-Euclidean.
We say that Heilbronn's criterion applies to a polynomial $f$
if it applies to the number field $K=\Q[x]/(f)$ generated by $f$.

Suppose $n\geq 3$ is odd and $p\geq 5$ is prime with
$\gcd(p-1,n)=1$.
Let $\mathcal{F}_{p,n}$ denote the collection
of monic polynomials $f\in\Z[x]$ of degree $n$
that are Eisenstein at the prime $p$.
We order our polynomials by the natural height $\Ht(f)$.
Define $\delta_{p,n}(X)$ to be the 
proportion of polynomials $f\in\mathcal{F}_{p,n}$
with $\Ht(f)\leq X$
for which Heilbronn's criterion applies.
One has
$$
\liminf_{X\to\infty}
\delta_{p,n}(X)\geq \max\left\{\frac{2}{27}\,,\;1-\eps(p)\right\}
\,,
$$
where $\eps(p)\to 0$ and is effectively computable.
In particular, the lower density tends to $1$ as $p\to\infty$
uniformly in $n$.
We also give a version of this result where we weaken the
condition on $\gcd(p-1,n)$.

As a corollary, we show that given an integer $n\geq 2$,
a positive proportion of Eisenstein polynomials of degree $n$
fail to generate norm-Euclidean fields.
\end{abstract}

\maketitle

\section{Introduction}

We are interested in proving statistical results concerning
polynomials with integer coefficents.  Let $\mathcal{F}\subseteq\Z[x]$.
Before proceeding, we must have an ordering on our set,
which is typically endowed by a height function.
Let $\mathcal{F}\subseteq\Z[x]$ be a collection of polynomials.
Most importantly, a height function $H:\mathcal{F}\to\R_{\geq 0}$ should satisfy the following property:
For any $X>0$, the set
$$
\{f\in\mathcal{F}\mid H(f)\leq X\}
$$
is finite.  
Given $$f(x)=a_n x^n+\dots+a_1 x+a_0\in\Z[x]\,,$$
we define the standard height function
$$
\Ht(f)=\max\{|a_0|,|a_1|,\dots,|a_n|\}
\,.
$$
One might consider other height functions, but here we will always
use $\Ht(f)$ as defined above.
Given a subset $\mathcal{F}\subseteq\mathbb{Z}[x]$, we define
the associated counting function
$$
  N(X)=\#\{f\in\mathcal{F}\mid \Ht(f)\leq X\}
  \,.
$$
Given a distinguished subset $\mathcal{F}^*\subseteq\mathcal{F}$,
we let $\delta(X)$ denote the proportion of elements of $\mathcal{F}$ that belong to $\mathcal{F}^*$; that is
$$
  \delta(X)=\frac{N^*(X)}{N(X)}
  \,,
$$
where we have written $N^*(X)$ to denote the counting function associated to $\mathcal{F}^*$.
One then defines the density of $\mathcal{F}^*$ in $\mathcal{F}$ to be
$$
\delta\colonequals \lim_{X\to\infty}\delta(X)
\,,
$$
provided the limit exists.
In cases where it is difficult to prove this limit exists, 
one might seek bounds on the lower and upper densities:
$$
\underline{\delta}\colonequals
\liminf_{X\to\infty}\delta(X)
\,,\;\; 
\overline{\delta}\colonequals
\limsup_{X\to\infty}\delta(X)
\,.$$

We begin with an important example.
Fix a positive integer $n$.
Denote by $\mathcal{F}_n$ the collection
of all monic polynomials in $\Z[x]$ of degree $n$.
It is easy to show that
$$
   N(X)\sim (2X)^n
   \,.
$$
Let $\mathcal{F}_n^*$ be the 
subset of $\mathcal{F}_n$ consisting of Eisenstein polynomials.
Dubickas considers the probability that a randomly chosen 
polynomial of degree $n$ is Eisenstein (see~\cite{MR1995564}).  Indeed, he
proves that the associated density exists and that it equals
\begin{equation}\label{E:dub}
  \delta=1-\prod_p\left(1-\frac{1}{p^n}+\frac{1}{p^{n+1}}\right)
  \,.
\end{equation}
This appears to be a topic of interest, as this paper led to
the papers~\cite{MR3063896, MR3278954, MR3544517, MR3893668}, among others.

On the other hand, one of the central objects of study in algebraic number theory
are number fields.  As an Eisenstein polynomial
$f\in\mathcal{F}^*_n$ is always irreducible,
it generates a number field $K$ of degree $n$.
That is, if we define $K=\Q[x]/(f)$ then one has $[K:\Q]=n$.
We can ask various questions about properties of
this number field as we vary the polynomial.
As is standard, we let $\O_K$ denote the ring of integers in $K$,
and let $N:K\to\Q$ denote the norm map.

One classical question is regarding when the Euclidean algorithm holds in $K$.
We call a number field $K$ norm-Euclidean if for every $\alpha,\beta\in\O_K$, $\beta\neq 0$,
there exists $\gamma,\rho\in\O_K$ such that $\alpha=\gamma\beta+\rho$
and $|N(\rho)|<|N(\gamma)|$.  This is the generalization of the
division property of $\Z$, which ensures that the Euclidean algorithm terminates after a finite number of steps.  The statement that $K$ is norm-Euclidean is equivalent to $\O_K$ being a Euclidean domain with respect to the function $\partial(\alpha)=|N(\alpha)|$.
The quadratic fields that are norm-Euclidean have been completely determined
(see~\cite{MR41883, MR41885, MR53162, MR1191702}),
but this has not been accomplished in full for any degree $n>2$.
The reader may be interested to consult the survey article~\cite{MR1362867}.

One knows that 
all norm-Euclidean fields have class number one, but the converse is not
true in general. Recall that $K$ having class number one is equivalent to $\O_K$ being a unique factorization domain.
We expect the norm-Euclidean property to be rather rare in general, but proving this is quite difficult.
On the other hand, if $K$ is generated by an Eisenstein polynomial, we actually have a tool available to us, due to the existence of a totally ramified prime.  More specifically, if $f$ is Eisenstein at the prime $p$, then $p$ is totally ramified in $K$; that is, $(p)=\mathfrak{p}^n$ for some prime ideal $\mathfrak{p}$.  The tool we refer to is as follows
(see~\cite{MR1574525, MR35313, MR43134, MR752120}):

\begin{lemma}[Heilbronn's criterion]\label{L:heilbronn}
Let $K$ be a number field of degree $n$.
Let $p$ be a prime that is totally ramified in $K$.
If we can write $p=a+b$ where $\{a,-b\}$ are not norms and $a$ is an $n$-th power residue modulo $p$, then $K$ is not norm-Euclidean.
\end{lemma}
For clarity, recall that we refer to $a\in \Z$ as a norm (from $\O_K$) when
there exists $\alpha\in \O_K$ such that $N(\alpha)=a$.
Note that if $n$ is odd, then $b$ is a norm if and only if $-b$ is norm and hence
one can simply write $\{a,b\}$ where we have written $\{a,-b\}$
in the lemma above.  Additionally, note that the criterion never applies
when $p=2$ or $p=3$, so we should always be considering $p\geq 5$ in our applications
of this result.

A natural question to ask is:  How often does Heilbronn's criterion apply to a number field $K$?
One could ask this question in the case of all number fields of fixed degree,
ordered by their discriminant.  
There are some results in this direction when
$n=3$ and $n=5$ (see~\cite{MR4613609, MR4649640}).
However, such questions are intractable 
at present when $n>5$ (see, for example, the introduction to~\cite{MR2745272}).
On the other hand, if we consider the corresponding
statistical questions for polynomials, then something can be said.  For convenience, we will say Heilbronn's criterion applies to a polynomial $f\in\Z[x]$ if the criterion applies to the number field
$K=\Q[x]/(f)$.  We establish the following result:

\begin{theorem}\label{T:1}
Suppose $n\geq 3$ is an odd integer and $p\geq 5$ is a prime with
$\gcd(p-1,n)=1$.
Let $\mathcal{F}_{p,n}$ denote the collection
of monic polynomials $f\in\Z[x]$ of degree $n$
that are Eisenstein at the prime $p$.
Define $\delta_{p,n}(X)$ to be the 
proportion of polynomials $f\in\mathcal{F}_{p,n}$
with $\Ht(f)\leq X$
for which Heilbronn's criterion applies.
One has
$$
\liminf_{X\to\infty}
\delta_{p,n}(X)\geq \max\left\{\frac{2}{27}\,,\;1-\eps(p)\right\}
\,.
$$
Here $\eps(p)\to 0$ and is effectively computable.
\end{theorem}

When $p$ is large, there are many ways to write $p=a+b$ and hence
determining the exact density of polynomials that satisfy Heilbronn's
criterion is likely very difficult.  However, the previous
theorem says that the lower density tends to $1$ as $p\to\infty$,
uniformly in the variable $n$.

As a result, we immediately obtain the following consequence:
When $n\geq 3$ is odd,
a positive proportion of Eisenstein polynomials of degree $n$
fail to generate norm-Euclidean fields.

The condition from Theorem~\ref{T:1} that $\mathrm{gcd}(p-1,n) = 1$
is not a severe restriction as there are infinitely 
many primes $p$ that satisfy this condition for each fixed $n$.
Nonetheless, we show that we can
relax this condition at the expense of considering larger primes.
This will allow us to deal with the case where $n$ is even.

\begin{theorem}\label{T:2}
Suppose $n\geq 2$ and $p\geq 5$ is a prime with
$$
\gcd(p-1,n)<\frac{p^{1/2}}{(\log p)^2}
\,.
$$
Define $\delta_{p,n}(X)$ as in Theorem~\ref{T:1}.
One has
$$
\liminf_{X\to\infty}
\delta_{p,n}(X)\geq 1-\hat{\eps}(p)
\,,
$$
where $\hat{\eps}(p)\to 0$ and is effectively computable.
\end{theorem}

\begin{corollary}\label{C:1}
Fix an integer $n\geq 2$.
A positive proportion of Eisenstein polynomials of degree $n$
fail to generate norm-Euclidean fields.
\end{corollary}

In \S\ref{S:EH} we discuss Eisenstein polynomials and Heilbronn's criterion,
and outline the strategy of our proof.  In addition, we establish some preliminary 
lemmas.  In the case where the proofs are very short, we err on the side of giving
proofs of known results for the sake of completeness.
In \S\ref{S:gotime} we prove our main counting result for Eisenstein polynomials
satisfying certain other local conditions.  In \S\ref{S:local} we prove a formula
for the relevant local densities, which is then used to establish the necessary bounds.
In \S\ref{S:final}
we put everything together and prove Theorem~\ref{T:1}. Finally, in \S\ref{S:reallyFinal}, we prove Theorem \ref{T:2}.

\section{Eisenstein polynomials and Heilbronn's criterion}\label{S:EH}

Let $f\in\Z[x]$ and write
$$
  f(x)=a_nx^n+a_{n-1}x^{n-1}+\dots+a_1x+a_0
  \,.
$$
We say that $f$ is Eisenstein at the prime $p$
(or $p$-Eisenstein for short) if 
$p \mid a_i$
for $i=0,\dots,n-1$, $p\nmid a_n$, and
$p^2\nmid a_0$.
We refer to a polynomial as Eisenstein, if it is
$p$-Eisenstein for some prime $p$.
It is well-known
that Eisenstein polynomials are irreducible in $\Z[x]$.
We will always assume $f$ is monic (i.e., $a_n=1$)
which makes the condition at $a_n$ unnecessary.

We will write $\Z[x]_{\text{mon}}$ for the collection
of monic polynomials with coefficents in $\Z$.
Throughout the rest of this section we will always
consider $f\in\Z[x]_\text{mon}$ and set $K=\Q[x]/(f)$.
The letters $p,q$ will always denote primes in $\Z$.

We outline our strategy.
Let $f\in\Z[x]_\text{mon}$ be a $p$-Eisenstein polynomial
with $p\geq 5$.
By Lemma~\ref{L:ramified} below, we know that $p$ is totally ramified in $K$.
This opens up the possibility of applying
Heilbronn's criterion (i.e., Lemma~\ref{L:heilbronn}).
We can simplify matters by assuming $\gcd(n,p-1)=1$,
which guarantees that every integer is an $n$-th power residue modulo $p$.
As mentioned earlier,
this will not be a severe restriction as there are infinitely 
many primes $p$ that satisfy this condition for each fixed $n$.

Given two primes $q_1<q_2<p$
we write
$p=uq_1+vq_2$ for $(u,q_1)=1$ and $(u,q_2)=1$.
This is possible when $p$ is large enough compared to $q_1$, $q_2$
(see Lemma~\ref{lem: Froby}).
If $q_1$ and $q_2$ are not norms from $K$,
then this is enough to invoke Heilbronn's criterion.
Moreover, the condition that $f(x)$ has no roots modulo $q$
is sufficient to guarantee that $q$ is not a norm
(see Lemma~\ref{L:key}).

\begin{lemma}\label{L:ramified}
If $f(x)$ is $p$-Eisenstein, then $p$ is totally ramified in~$K$.
\end{lemma}

\begin{proof}
Let $\mathfrak{p}$ be a prime in $K$ over $p$,
and set $e= v_\mathfrak{p}(p)$.
We aim to show $e=n$.  Of course, a priori one has $e\leq n$.
Let $\alpha$ be a root of $f$ in $K$.  From $f(\alpha)=0$,
using the $p$-Eisenstein condition,
we find $\mathfrak{p}$ divides $(\alpha)$, and therefore $v_\mathfrak{p}(\alpha^n)\geq n$.
Moreover,
$v_\mathfrak{p}(a_0)=e$ and $v_\mathfrak{p}(a_i\alpha^i)\geq e+1$ for $i=1,\dots,n-1$,
which implies $v_\mathfrak{p}(\alpha^n)=e$.  It follows that
$n\leq e$, completing the proof.
\end{proof}


\begin{lemma}\label{lem: Froby}
    Let $q_1 < q_2$ be primes. Let $p \ge q_1^2q_2^2$. Then there exist $u,v\in\Z^+$ such that
    \begin{enumerate}
        \item $p = uq_1 + vq_2$,
        \item $q_1\nmid u$, $q_2\nmid v$. 
    \end{enumerate}
\end{lemma}
\begin{proof}
Let $r_1\in \{1,2,\ldots, q_1-1\}$ and $r_2\in \{1,2,\ldots, q_2-1\}$ be fixed. Then $p\ge q_1^2q_2^2$ implies
$$p-r_1q_1-r_2q_2 \ge q_1^2q_2^2 -(q_1-1)q_1-(q_2-1)q_2 \ge (q_1^2-1)(q_2^2-1).$$

By the Frobenius coin-exchange problem, since $q_1^2$ and $q_2^2$ are coprime, and 
$$p-r_1q_1-r_2q_2\ge (q_1^2-1)(q_2^2-1),$$
there exist nonnegative integers $t_1, t_2$ such that
$$p-r_1q_1-r_2q_2 = q_1^2t_1+q_2^2t_2.$$
Then
$$p = q_1(q_1t_1+r_1)+q_2(q_2t_2+r_2) = q_1u+q_2v,$$
where $u = q_1t_1 + r_1$ and $v = q_2t_2 + r_2$ with $q_1\nmid u$ and $q_2\nmid v$.
\end{proof}


\begin{lemma}\label{L:key}
Let $q$ be prime.
If $q$ is a norm from $K$, then $f(x)$ has a root in $\F_q$.
\end{lemma}

\begin{proof}
Suppose $q$ is a norm.  Then there is a prime $\mathfrak{q}$ in $K$
such that $N(\mathfrak{q})=q$.  This means $|\O_K/\mathfrak{q}|=q$
and hence $\O_K/\mathfrak{q}\cong \F_q$.
Let $\alpha$ be a root of $f(x)$ in $K$
and consider $\overline{\alpha}$, the image of $\alpha$ under the 
reduction map $\O_K\to\O_K/\mathfrak{q}$.  Since $f(\alpha)=0$,
we must have $f(\overline{\alpha})=0$ in $\F_q$,
which proves the result.
\end{proof}

\begin{proof}[Proof of Lemma~\ref{L:heilbronn}]
Suppose $p$ is totally ramified and we can write $p=a+b$ as in
the statement of the result.
By way of contradiction, suppose $K$ is norm-Euclidean.
By hypothesis, $(p)=\mathfrak{p}^n$, where $N(\mathfrak{p})=p$.
It follows that $p=u\pi^n$ in $K$,
for some $u\in\mathcal{O}_K^\times$ and some
$\pi\in\mathcal{O}_K$ irreducible with $N(\pi)=p$.
Since $a$ is an $n$-th power residue modulo $n$, there exists $x\in\Z$ such that $x^n\equiv a\pmod{p}$.
Using the norm-Euclidean condition, we find
$x=\gamma\pi+\rho$ for some $\gamma,\rho\in\mathcal{O}_K$
with $|N(\rho)|<|N(\pi)|=p$.  We have $x\equiv \rho\pmod{\pi}$
and hence $x\equiv \rho^\sigma\pmod{\pi}$ for any
embedding $\sigma:K\to\C$.  This leads to
$a\equiv x^n\equiv N(\rho)\pmod{p}$.
We have $0<a<p$ and $-p<N(\rho)<p$.
Thus either $N(\rho)=a$ or $N(\rho)+p=a$.
This implies one of $\{a,-b\}$ is a norm, 
a contradiction.
\end{proof}


\section{Counting $p$-Eisenstein polynomials with local specifications}\label{S:gotime}

Given $a=(a_0,\dots,a_{n-1})\in\Z^n$, we define
$$
\Ht(a)=\max\{|a_0|,\dots,|a_{n-1}|\}
$$
so that it agrees  with our polynomial height.
The next result is what allows us to count polynomials 
of bounded height with local specifications.
This type of result is certainly well-known,
but the proof is short and so we include it for the 
sake of completeness.
\begin{lemma}\label{L:sieve}
Let $S$ be a subset of $(\Z/m\Z)^n$.  We have
$$
\#\{a\in\Z^n\mid a\hspace{-1ex}\mod{m}\in S\,,\;\Ht(a)\leq X\}
=
\frac{|S|}{m^n}(2X)^n+O\left(n2^n m(2X)^{n-1}\right)
\,.
$$
\end{lemma}

\begin{proof}
Let $\mathcal{A}\colonequals [-X,X]^n\cap\Z^n$.
Choose the maximal $k\in\Z^+$ such that
$$
  \mathcal{B}\colonequals (-km,km]^n\cap\Z^n\subseteq \mathcal{A}
  \,.
$$
The set $\mathcal{B}$ is chosen so that
$$
\#\{a\in\mathcal{B}\mid a\hspace{-1ex}\mod{m}\in S\}
=
|S|(2k)^n
$$
Using the Binomial Theorem,
one has
$$
\left(\frac{X}{m}\right)^n \ge k^n \ge \left(\frac{X}{m}-1\right)^n
=
\left(\frac{X}{m}\right)^n+O\left(n2^n\left(\frac{X}{m}\right)^{n-1}\right)
\,,
$$
and therefore
$$
\#\{a\in\mathcal{B}\mid a\hspace{-1ex}\mod{m}\in S\}=
\frac{|S|}{m^n}(2X)^n+O\left(n 2^n m(2X)^{n-1}\right)
\,.
$$
Next one bounds
$$
  \#(\mathcal{A}\setminus\mathcal{B})
  \leq
  (2X+1)^n-(2km)^n
  \ll
  n2^n m(2X)^{n-1}
  \,.
$$
The result follows.
\end{proof}

\begin{remark}\label{R:sieve}
One can think of the quantity $|S|/m^n$ in the statement of the previous lemma
as a density in $(\Z/m\Z)^n$,
and using the Chinese Remainder Theorem,
we can recover it as a product of local densities 
at the prime powers dividing $m$.
\end{remark}

Our typical application of
Lemma~\ref{L:sieve} will be to impose a condition
modulo $q_1$ and $q_2$, as well as the $p$-Eisenstein condition, which is a
condition modulo $p^2$.  In this case, the modulus is given as $m=q_1 q_2 p^2$.

Before proceeding, we require a couple definitions.
First we define
\begin{equation}\label{E:eisenstein}
E_p(n)\colonequals
\frac{1}{p^n}-\frac{1}{p^{n+1}}
\,,
\end{equation}
which is the $p$-Eisenstein density that contributes
to (\ref{E:dub}).
Secondly, we write $C_q(n)$ to denote --- the number of
$(a_0,\dots,a_{n-1})\in(\Z/q\Z)^n$ such that the corresponding polynomial
in $\F_q[x]$ has no roots in $\F_q$ ---
divided by $q^n$.
We are now ready to prove the following result, which
will be used in the proof of Theorem~\ref{T:1}.
\begin{prop}\label{P:count}
Let $n\in\Z^+$.
Suppose $q_1<\dots<q_t<p$ are primes.
Let $\mathcal{F}_{p,n}$ denote the collection
of monic polynomials $f\in\Z[x]$ of degree $n$
that are Eisenstein at the prime $p$.
The number of polynomials $f\in\mathcal{F}_{p,n}$
with $\Ht(f)\leq X$
for which
$f$ has no roots in $\F_q$ for all $q\in\{q_1,\dots,q_t\}$
equals
$$
  E_p(n)\left(\prod_{i=1}^t C_{q_i}(n)\right)(2X)^n+O\left(n 2^n q_1\dots q_t p^2 (2X)^{n-1}\right)
  \,.
$$
\end{prop}

\begin{proof}
This is an application of Lemma~\ref{L:sieve} with the modulus
$m=q_1\dots q_t p^2$.  By Remark~\ref{R:sieve}
we can determine the density $|S|/m^n$ for each prime power separately.
The condition that $f$ is $p$-Eisenstein corresponds to
\begin{align*}
&a_0\hspace{-1ex}\mod{p^2}\in\{p,2p,\dots,(p-1)p\}\,,
\\
&a_i\hspace{-1ex}\mod{p^2}\in\{0,p,\dots,(p-1)p\}\,,\;i=1,\dots,p-1
\,.
\end{align*}
Therefore the density mod $p^2$ equals $(p-1)p^{n-1}/p^{2n}$
in agreement with (\ref{E:eisenstein}).
The values of $C_q(n)$ are correct tautologically.
\end{proof}

\begin{remark}
In Proposition~\ref{P:count}, if one wants to specify that for certain
primes $q$, the condition described in the statement does not hold, then
one simply replaces $C_{q}(n)$ with $1-C_{q}(n)$.
\end{remark}


\section{Evaluation and estimation of the local densities}\label{S:local}

The following lemma gives a formula for our local density
$C_p(n)=A_p(n)/p^n$.  Here $C_p(n)$ is the quantity
defined immediately before Proposition~\ref{P:count}
and $A_p(n)$ is defined forthwith.

\begin{lemma}\label{lem: A_p formula}
The number of monic polynomials in $\F_p[x]$ of degree $n$
with no roots equals
\begin{equation}\label{E:Fg}
A_p(n)\colonequals
\sum_{\substack{f\in \mathbb{F}_p[x]_{\text{mon}}\\f \text{ has no roots}\\
    \deg(f) = n}}1
    =
\sum_{k=0}^{n}(-1)^k {p\choose k}p^{n-k}
\,.
\end{equation}
  When $n\geq p$, this becomes
$$
\left(
1-\frac{1}{p}
\right)^p
p^n
\,.
$$
\end{lemma}
\begin{proof}
    The count proceeds by inclusion-exclusion. We start with all polynomials, then remove the polynomials that are multiples of $(x-a)$ for $a\in\mathbb{F}_p$. Therefore, we get that the number of monic polynomials in $\mathbb{F}_p[x]$ of degree $n$ with no roots equals
    $$\sum_{k=0}^n (-1)^k{p\choose k}p^{n-k}.$$    
    Note that as ${p\choose k}=0$ whenever $k>p$,
    we find that for $n\geq p$, one has
    \begin{align*}\sum_{k=0}^n (-1)^k{p\choose k}p^{n-k}&=\sum_{k=0}^p (-1)^k{p\choose k}p^{n-k} \\&= p^n \sum_{k=0}^p {p\choose k} \left(\frac{-1}{p}\right)^k = p^n\left(1-\frac{1}{p}\right)^p.
    \end{align*}
\end{proof}

Finally, we require bounds on $C_p(n)$.
\begin{lemma}\label{L:Cbounds}
For $n\ge 2$ we have
\begin{equation}\label{eq: A ineq}
\frac{1}{4} \le \frac{p^2-1}{3p^2}\le C_p(n) \le \frac{p-1}{2p} < \frac{1}{2}\,. 
\end{equation}
\end{lemma}

\begin{proof}
First note that 
    $f(k) = {p\choose k}/p^k$
    is a decreasing function. Indeed, for $k\le p-1$
    $$\frac{f(k)}{f(k+1)} = \frac{(k+1)p}{p-k} > 0,$$
    so $f(k) > f(k+1)$ for all $k\le p-1$. Since $f(k) = 0$ for $k > p$, and $f(k) > 0$ for $k\le p$, it follows that $f(k)$ is decreasing. Since $f(k)$ is decreasing and $A_p(n)$ is an alternating sum of positive decreasing terms, we get for $n\ge 2$,
    $$\frac{{p\choose 2}}{p^2} - \frac{{p\choose 3}}{p^3} \le \frac{A_p(n)}{p^n} \le \frac{{p\choose 2}}{p^2}.$$
    \end{proof}

\section{Proof of the main result}\label{S:final}

\begin{proof}[Proof of Theorem~\ref{T:1}]
Adopt the notation from the statement of the theorem.
Further, write $N_{p,n}(X)$ to denote the number of 
$f\in\mathcal{F}_{p,n}$ with $\Ht(f)\leq X$.
We have $N_{p,n}(X)\sim E(p)(2X)^n$ where $E(p)$
is given in (\ref{E:eisenstein}).
This follows from~\cite{MR1995564} or alternatively,
from Proposition~\ref{P:count} with $t=0$.

First we will establish the lower bound
$\liminf_{X\to\infty}\delta_{p,n}(X)\geq 2/27$.
Assume $p\neq 7,11,19$.
By Lemma~\ref{lem: Froby}, we can write $p=2u+3v$ with
$(2,u)=1$ and $(3,v)=1$ for $p\ge 36$. One can then manually check that the only examples of primes up to 36 that cannot be expressed in such a way are $p = 7, 11, 19$.
Therefore, provided $f(x)$ has no roots in $\F_2$
or $\F_3$, Heilbronn's criterion applies.
(Recall Lemma~\ref{L:heilbronn} and the discussion in~\S\ref{S:EH}.)
Write $N^*_{p,n}(X)$ to denote the counting function for this collection
of polynomials.
Using Proposition~\ref{P:count}
with $q_1=2$, $q_2=3$, we have
$$
  N^*_{p,n}(X)\sim E_p(n)C_2(n)C_3(n)(2X)^n
  \,,
$$
and therefore,
via Lemma~\ref{lem: A_p formula} with $n\geq 3$,
we find
$$
\liminf_{X\to\infty}
\delta_{n,p}(X)\geq C_2(n)C_3(n)= (1/4)(8/27)=2/27
\,.
$$

In the case where $p\in\{7,11,19\}$,
we can write $p=2u+5v$ with
$(2,u)=1$ and $(5,v)=1$.
The same approach proves
$$
\liminf_{X\to\infty}\delta_{n,p}(X)\geq C_2(n)C_5(n)\geq(1/4)(8/25)>2/27
\,.
$$
Indeed, by Lemma~\ref{L:Cbounds},
as $n\geq 3$ is odd,
one has $C_5(n)\geq 8/25$.

It remains to establish a lower bound $\delta_{p,n}(X)\geq 1-\eps(p)$
where $\eps(p)\to 0$.
We may assume that $p$ is large.
Choose $Y$ maximal such that \mbox{$Y^4<p$}.
Let $q_1<\dots<q_t\leq Y$ be all the primes less than $Y$.
In particular, Lemma~\ref{lem: Froby} applies
to any pair of primes in $Q\colonequals\{q_1,\dots,q_t\}$
and we have $t=\pi(Y)\leq \pi(p^{1/4})$.
Before proceeding with the final estimate,
observe that by Lemma~\ref{L:Cbounds}, one has
$$
  D\colonequals\max_{j=1,\dots,t}\left\{\frac{C(q_j)}{1-C(q_j)}\right\} \le 1
  \,.
$$
We wish to bound from above the
number of $p$-Eisenstein polynomials with $\Ht(f)\leq X$
such that Heilbronn's criterion does not apply for any
pair of primes in $Q$.
This means that $f$ has no root modulo $q$ for at most one $q\in Q$.
Using Proposition~\ref{P:count}
in the situation where
$\{q\in Q\mid f \text{ has a root in }\F_q\}$
equals $Q$, $Q\setminus\{q_1\}$, $Q\setminus\{q_2\}$, and so on,
we obtain
\begin{align}\label{E:liminf}
  &1-\liminf_{X\to\infty}\delta_{n,p}(X)
  \notag\\
  &\quad\leq
  \prod_{i=1}^t(1-C(q_i))+
  \sum_{j=1}^tC(q_j)\left((1-C(q_j))^{-1}\prod_{i=1}^t(1-C(q_i)\right)
  \notag\\
  &\quad\leq
  (1+tD)\prod_{i=1}^t(1-C(q_i))
  \notag\\
  &\quad\leq
  \left(1+\pi\left(p^{1/4}\right)\right)\left(\frac{3}{4}\right)^{\pi(p^{1/4})}=:\eps(p)\,.
\end{align}
In the last step, we have used
the fact that $C_p(n)\ge 1/4$ for $p\geq 2$
which follows from Lemma~\ref{L:Cbounds}.
\end{proof}

\section{Weakening the gcd condition}\label{S:reallyFinal}

\begin{proof}[Proof of Theorem~\ref{T:2}]

Let $Y>0$ be a parameter to be chosen later.
Let $\mathcal{F}^*_{p,n}$ denote the subset of 
$\mathcal{F}_{p,n}$ consisting of polynomials 
for which there exist primes
$q_1<q_2\leq Y$ that are not norms in $K$.
Applying the argument from the proof of Theorem~\ref{T:1}
leading up to (\ref{E:liminf}), we find that the
lower density of $\mathcal{F}^*_{p,n}$ in $\mathcal{F}_{p,n}$
is bounded below by $1-\hat{\eps}(p)$ where
$$
\hat{\eps}(p)\colonequals \left(1+\pi(Y)\right)\left(\frac{3}{4}\right)^{\pi(Y)}
\,.
$$

Set $g\colonequals \gcd(n,p-1)$.
Next we show that Heilbronn's criterion applies
when $g$ is small enough compared to $p$.
Let $f\in\mathcal{F}_{p,n}^*$.  Then there exist primes
$q_1<q_2\leq Y$ that are not norms in $K$.
Let $\psi$ be a primitive Dirichlet character modulo $p$
of order $g$, so that $\psi(u)=1$ if and only if $u$ is an $n$-th power residue modulo $p$.
Note that there are $(p-1)/g$ elements of $(\Z/p\Z)^\times$
that are $n$-th power residues modulo $p$.

Set $X\colonequals p/q_1-2q_2$.
Assume $Y<p/4$ so that $X<p$.
For the moment,
assume we can find a positive integer $u<X$ such that
\begin{equation}\label{E:new}
(u,q_1)=1\,,
\;\;
q_1 u\equiv p\pmod{q_2}\,,
\;\;
\psi(uq_1)=1
\,.
\end{equation}
Since $u<p/q_1$, this implies we can write
$p=uq_1+vq_2$ with $v>0$.  Observe that
Heilbronn's criterion applies,
except possibly when $q_2\mid v$.  In the case where $q_2\mid v$,
we try $p=(u+q_2)q_1+(v-q_1)q_2$, which works
unless $q_1\mid u+q_2$.  If this is the case, we try
$p=(u+2q_2)q_1+(v-2q_1)q_2$ which is forced to work.
Note that $u<X$ implies $u<u+q_2<u+2q_2<p/q_1$ and hence
the aforementioned expressions all involve positive integers.

Let us count the number of
$u<X$ satisfying (\ref{E:new}).  Using standard techniques,
this count equals
\begin{equation}\label{E:main.error}
  X\frac{q_1-1}{q_1}\cdot\frac{1}{q_2}\cdot\frac{1}{g}+O\left(m^{1/2}\log m\right),
\end{equation}
where the error term comes from an application of the P\'olya--Vinogradov
inequality to a Dirichlet character of modulus $m=p q_1 q_2$.
Indeed, one sums the product of the following indicator functions over $u < X$:
\begin{align*}
\phi_0(u)
&=
\begin{cases}1 & \text{ if } \mathrm{gcd}(u,q_1)=1\,,\\0 & \text{ if }\mathrm{gcd}(u,q_1)>1\,,\end{cases}
\\
\frac{1}{\phi(q_2)}
\sum_{\chi\text{ mod }q_2}
\chi(uq_1)\chi(p)^{-1}
&=
\begin{cases}1 & \text{ if } q_1 u\equiv p\pmod{q_2}\,,\\0 & \text{otherwise}\,,\end{cases}
\\
\frac{1}{g}\sum_{k=1}^g\psi^k(uq_1)
&=
\begin{cases}1 & \text{ if } \psi(uq_1)=1\,, \\ 0 & \text{otherwise}\,.\end{cases}
\end{align*}
In the above $\phi_0$ is the principal character modulo $q_1$,
the sum $\sum_{\chi\text{ mod }q_2}$ is taken over all Dirichlet characters
modulo $q_2$, and $\psi$ is the character modulo $p$ defined above.

In order to guarantee that $u<X$ satisfying (\ref{E:new}) exists, we must have the main term
of (\ref{E:main.error})
strictly greater than the error term.
Using $q_1\geq 2$ and the definition of $X$,
the following condition suffices
$$
  \frac{1}{2}\left(\frac{p}{q_1 q_2}-2\right)\frac{1}{g}>Cm^{1/2}\log m
  \,,
$$
where $C$ is the constant from P\'olya--Vinogradov.
Using $m=p q_1 q_2$ with $q_1,q_2\leq Y$, and some simple manipulation,
the sufficient condition becomes
$$
  g<\frac{p^{1/2}}{2C Y^3\log(pY^2)}
  -\frac{1}{C p^{1/2}\log{p}}
  \,.
$$
It suffices to prove the theorem when $p$ is sufficiently large,
and the result follows for large $p$ upon choosing $Y=(\log p)^{1/4}$.
\end{proof}

\begin{proof}[Proof of Corollary~\ref{C:1}]
First suppose $n$ is odd.
From Theorem~\ref{T:1}, we have
$$
  \liminf_{X\to\infty}\delta_{5,n}(X)\geq 2/27
  \,.
$$
Since $\mathcal{F}_{5,n}$ is of positive density inside $\mathcal{F}_n$,
the set of all monic $f\in\Z[x]$ with $\deg(f)=n$,
the result follows.

Next, suppose $n$ is even.  We would like to invoke
Theorem~\ref{T:2}.  We know there is a $p_0$ such that
$\hat{\eps}(p)<1$ and
$2<\sqrt{p}/(\log{p})^2$
when $p\geq p_0$.
To invoke the theorem, we need a prime
$p\geq p_0$ such that $\gcd(p-1,n)$ is small enough.
We claim there exists infinitely many primes $p$
satisfying the condition.  Indeed, when $p\equiv -1\pmod{2n}$
one finds $\gcd(p-1,n)=2$.  Choosing such a prime with $p\geq p_0$
allows us to conclude that
$$
\liminf_{X\to\infty}\delta_{p,n}(X)>0
\,,
$$
completing the proof.
\end{proof}

\section{Acknowledgements}
This research was conducted as part of the Research Experience for Undergraduates and
Teachers program at California State University, Chico, supported by NSF grant DMS2244020.
The authors would like to thank Paul Pollack for helpful conversations, including suggesting Theorem \ref{T:2}.

\bibliography{heilbronn}{}
\bibliographystyle{plain}


\end{document}